\documentclass[11pt]{article}

\usepackage{amsmath,amssymb,amsthm}
\usepackage[utf8]{inputenc}

\usepackage[english]{babel}

\usepackage[pdftex]{graphicx}
\DeclareGraphicsExtensions{.png,.pdf,.jpg}
\usepackage{color}

\usepackage[colorlinks=false,hidelinks]{hyperref}

\usepackage{multicol}

\usepackage{subcaption}

\newtheorem{lemma}{Lemma}
\newtheorem{rem}{Remark}
\newtheorem{theo}{Theorem}
\newtheorem{cor}{Corollary}

\newcommand{\N}{\mathbb{N}}
\newcommand{\Z}{\mathbb{Z}}
\newcommand{\C}{\mathbb{C}}

\newcommand{\e}{\mathrm{e}}
\newcommand{\dd}{\mathrm{d}}

\DeclareRobustCommand{\stirling}{\genfrac\{\}{0pt}{}}

\title{Asymptotics and zero behaviour of Geometric polynomials}

\author{M. Bello-Hernández \and M. Benito \and Ó. Ciaurri \and E. Fernández}

\begin{document}



\maketitle

\date{}

%


\abstract{ 
We obtain some results on the asymptotic behaviour of Geometric polynomials in both the complex plane minus $[-1,0]$ and the interval $(-1,0)$.  We also find the distance of consecutive zeros of these polynomials in the bulk of the interval $(-1,0)$. We also prove that they satisfy certain orthogonality properties. Due to their relationship with Eulerian polynomials, the results for Geometric polynomials can be transferred to Eulerian polynomials verbatim.
}

\vspace{0.5cm}
\noindent
\textbf{Keywords:}\\
\hspace*{1cm}Geometric polynomials, Eulerian polynomials, Asymptotics, Zero distribution, Orthogonality.

\vspace{0.5cm}
\noindent
\textbf{MSC Class:}\\
\hspace*{1cm}11B83, 33C47, 41A46

\section{Introduction}
Eulerian and Geometric polynomials  play a significant role in combinatorics, special functions, and analytic number theory. They form a distinguished sequence of polynomials characterized by their close connection to the structure of set partitions, making them a useful tool for encoding combinatorial information in algebraic form. 

The Geometric polynomials $(P_n(z))_{n\ge 0}$ are given by
\begin{equation}
\label{mainEquation}
f(z,t)=\frac{1}{1-z(e^t-1)}=\sum_{n=0}^\infty P_n(z)\frac{t^n}{n!},
\end{equation}
meanwhile Eulerian polynomials $(A_n(z))_{n\ge 1}$  are  defined by
\[
\sum_{k=1}^\infty k^nz^k=\frac{zA_n(z)}{(1+z)^n}.
\]
These functional identities determine uniquely  both sequence 
of polynomials and provide a convenient framework for deriving explicit formulas, recurrence relations, and structural properties of these polynomials. See for example \cite{Ade_Ben_Nko}, \cite{Boy_Dil}, \cite{Dur}, \cite{Gra_Knu_Pat}, \cite{Mez}, \cite{Tan}. 

 The polynomials $(P_n(z))_{n\ge 0}$  can be given in term of Stirling number of second kind $\stirling{n}{k}$, namely
\begin{equation}
\label{eqPolStirling}
P_n(z)=\sum_{k=1}^{n}k!\stirling{n}{k}z^k,\quad n=1,2,\ldots,
\end{equation}
since $\frac{(e^t-1)^k}{k!}=\sum_{j=k}^\infty \stirling{j}{k}\frac{t^j}{j!}$. Stirling number of second kind  stands for the number of ways to partition a set of $n$ objects into $k$ non-empty subsets. Euler's work \cite[p. 485]{Eul} seems to be where these polynomials first appear.  Geometric polynomials are also called Fubini polynomials because Fubini numbers are given by
\begin{equation}
\label{FubNum}
P_n(1)=\frac12\sum_{k=0}^\infty\frac{k^n}{2^k},\quad n=0,1,\ldots
\end{equation}
These numbers enumerate various combinatorial objects, including ordered partitions, which are partitions that also take the order of blocks into account. See  for example \cite{Fla_Sed}  and \cite{Gro}.

The Eulerian polynomials are given  by $A_0(z)=1$,
\[
A_n(z)=\sum_{k=0}^{n-1} \left\langle \begin{matrix} n \\ k \end{matrix} \right\rangle z^k,\quad n=1,2,\ldots,
\]
where $\left\langle \begin{matrix} n \\ k \end{matrix} \right\rangle$ are the Eulerian numbers. See  for example \cite[p. 51]{Com}. It is well-known that Eulerian numbers have also a combinatorial interpretation in term of permutations. In fact, the Eulerian number $\left\langle \begin{matrix} n \\ k \end{matrix} \right\rangle$  counts the number of permutations the number of permutations
in $\{1,2,\ldots,n\}$ with $k$ descents; so they satisfy the recurrence
\[
\left\langle \begin{matrix} n+1 \\ k \end{matrix} \right\rangle=(n-k+2)\left\langle \begin{matrix} n \\ k-1 \end{matrix} \right\rangle+k\left\langle \begin{matrix} n \\ k \end{matrix} \right\rangle,\quad n=0,1,\ldots, \quad k\ge 2,
\]
assuming $\left\langle\begin{matrix} n \\ 1 \end{matrix} \right\rangle=1$ for $n\ge 0$ and $\left\langle\begin{matrix} 0 \\ k \end{matrix} \right\rangle=0$ for $k\ge 2$. Eulerian and Geometric polynomials are linked by the formula
\begin{equation}
\label{conecEulerianPolvsGeomPol}
P_n\left(\frac{1}{z-1}\right)=\frac{A_n(z)}{(z-1)^n},\quad n\ge 1,
\end{equation}

Our objective  is to study the asymptotic behaviour of Eulerian and Geometric polynomials.  For example, we state a result on the asymptotic behaviour of Geometric polynomials in the interval $(-1,0)$. The corresponding results for Eulerian polynomials and the asymptotic in the complex plane minus $[-1,0]$ will be formulated later. These theorems will be proved in Section \ref{SectProofMain-Result}.  
Since the zeros of the polynomials $(P_n(z))_{n\in\N}$ are at $(-1,0)$, the Geometric polynomials exhibit an oscillatory behaviour in the interval $(-1,0)$. The following result describes this behaviour.
To simplify writing for $z\in (-1,0)$, denote $\ell(z)=\log|\frac{1+z}{z}|$
and 
\[
J_n(z)=\frac{(\ell(z)+\pi i)^{n+1}+(\ell(z)-\pi i)^{n+1}}{((\ell(z))^2+\pi^2)^{(n+1)/2}},\quad n\in \N.
\]

\begin{theo}  \label{teoAsympIn}  For $c\in\big(0,\tfrac{\pi}{\sqrt{2}}\big)$, let
\begin{equation}
\label{DefSeq}
 L_n=c\sqrt{\frac{n}{\log n}},\quad z_n=-\frac{e^{L_n}}{1+e^{L_n}}, \quad n\ge 3.
\end{equation}
Then  
\[
\lim_{n\to\infty}\left(\frac{(1+z)P_n(z)((\ell(z))^2+\pi^2)^{(n+1)/2}}{n!}-J_n(z)\right)=0
\]
uniformly for $z\in[z_n,-1-z_n]$.
\end{theo}

Taking a closer look at Theorem \ref{teoAsympIn} allows us to determine the distance between the zeros of Geometric polynomials. To achieve this, we number the zeros of $P_n(z)$ with respect to $-1/2$. The zeros of $P_n$ are denoted by $z_{j,n}$ for $j$ positive or negative, depending on whether they lie to the right or left of $z = -1/2$, starting from the zero nearest to $z = -1/2$ and taking into account that $z_{-j,n}+z_{j,n}=-1$. As $P_{2k}(-1/2)=0$ and $P_{2k+1}(-1/2)\ne 0$ we have to distinguish indexes even from odd. The following result then follows.

\begin{theo}\label{teoAproxZerosBulk}  We have
\begin{equation}
  z_{j,2n}=-\frac12 + \frac{\pi^2}{4}\,\frac{j}{n+1}  
  +\ o\!\Big(\frac{1}{n}\Big)\quad \text{as}\quad n\to\infty,
  \label{eq:bulkEven}
\end{equation}
for each $j=0, 1,\ldots$ While
\begin{equation}
  z_{j,2n+1}=-\frac12+ \frac{\pi^2}{8}\,\frac{2j-1}{n+1}  
  +\ o\!\Big(\frac{1}{n}\Big)\quad \text{as}\quad n\to\infty,
  \label{eq:bulkOdd}
\end{equation}
for each $j= 1, 2,\ldots$ For negative index the corresponding zeros are define by the symmetric relation $z_{-j,n}+z_{j,n}=-1$ . The distance of consecutive zeros of $P_n$ in the bulk\footnote{The bulk of the interval $(-1,0)$ is understood as any interval $(-1-a,a)$ with $a\in(-1/2,0)$ with $a$ fixed.} of $(-1,0)$ is of order $\frac{\pi^2}{4n}+o\!\Big(\frac{1}{n}\Big)$ as $n\to\infty$. 
\end{theo}

The above theorem will be proved in Section \ref{sec:Zeros} together with the corresponding for Eulerian polynomials. In Section \ref{sec:algebraic} some algebraic results for these polynomials are stated.

\section{Algebraic relations}\label{sec:algebraic}
In this section we set some algebraic properties of Geometric and Eulerian polynomials. The polynomial $P_n$ has exactly degree $n$, integers coefficients  and its leading coefficient is $n!$, $P_n(0)=0$, $P_n'(0)=1$ for all $n\ge 1$, and $P_n(-1)=(-1)^n.$ Equation \eqref{mainEquation} implies $1=(1-z\sum_{k=1}^\infty\frac{t^k}{k!})\sum_{k=0}^\infty\frac{P_k(z)t^k}{k!}$, so  we have $P_0(x)=1$ and
\begin{equation}
\label{Recurrencia}
P_n(z)=z\sum_{k=0}^{n-1}\binom{n}{k}P_k(z),\quad n=1,2,\ldots
\end{equation}
The first seven Geometric polynomials are given by
\[
P_0(x)=1,\quad P_1(x)=x,\quad P_2(x)=2x^2 + x,
\]
\[
P_3(x)=6x^3 + 6x^2 + x,\quad P_4(x)=24x^4 + 36x^3 + 14x^2 + x,
\]
\[
P_5(x)=120x^5 + 240x^4 + 150x^3 + 30x^2 + x,
\]
\[
P_6(x)= 720x^6 + 1800x^5 + 1560x^4 + 540x^3 + 62x^2 + x.
\]
In Figure \ref{figura} we show the graphics of polynomials $P_n(x)$ for $n=0,1,\ldots,6$. These graphs lead us to suspect that the roots of the polynomials are simple and lie within the interval $(-1,0]$. The next result states the location of the roots of $P_n$.
\begin{figure}
\centering
\includegraphics[width=0.6\textwidth]{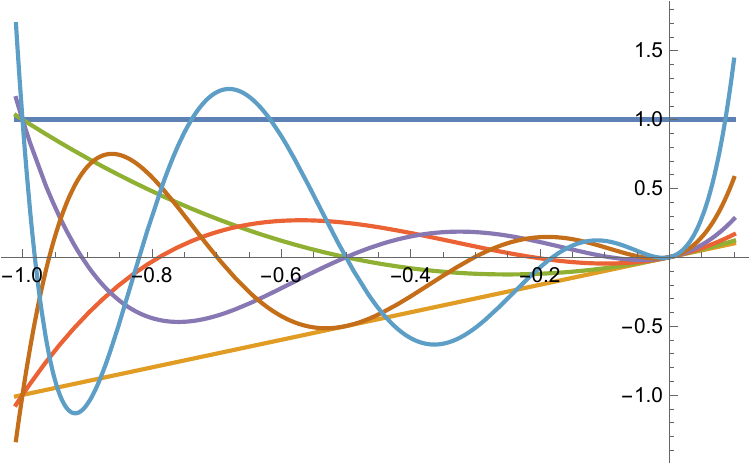}
 \caption{Graphics of $P_n(z)$  for $n=0,1,\ldots,6$ in the interval $[-1.01,0.1]$.}\label{figura}
 \end{figure}

\begin{lemma}\label{lemZer} 
\begin{enumerate}
\item
The following formula holds
\begin{equation}
\label{delDifEc}
P_n(z)=z\frac{d((1+z)P_{n-1}(z))}{dz}=z\,P_{n-1}(z)+z(1+z)\,P_{n-1}'(z),
\end{equation}
and $P_n$ has simple zeros in $(-1,0]$. Moreover the zeros of $P_n$ and $P_{n-1}$  different from $z=0$ interlace.

\item We have
\begin{equation}
\label{simtRel}
(1+z)P_n(z)=(-1)^nzP_n(-1-z),\quad n\ge 1.
\end{equation}
The zeros of $P_n$ different from $z=0$ are symmetric with respect to $z=-1/2$. If $n$ is even, $P_n$ has  a zero at $z=-1/2$. In addition
\begin{equation}
\label{RelPolDer}
(1+z)P_n'(z)+(-1)^nzP_n'(-1-z)=\frac{P_n(z)}{z}.	
\end{equation}
If $P_n(z)\ne 0$ and $z\ne -1$, it holds
\begin{equation}
\label{RelPolDer2}
\frac{P_n'(z)}{P_n(z)}+\frac{P_n'(-1-z)}{P_n(-1-z)}=\frac{1}{z(1+z)}.
\end{equation}
\end{enumerate}
\end{lemma}

\begin{proof} 1. As
\[
\stirling{n+1}{k}=k\stirling{n}{k}+\stirling{n}{k-1},
\]
from \eqref{eqPolStirling} we obtain
\[
P_n(z)=\sum_{k=1}^{n-1}k!k\stirling{n-1}{k}z^k+\sum_{k=2}^{n}k!\stirling{n-1}{k-1}z^k\\=z\frac{dP_{n-1}(z)}{dz}+z\frac{d(zP_{n-1}(z))}{dz},
\]
and \eqref{delDifEc} follows.

The statement about the zeros of $P_1(z)=z$ is trivial. From \eqref{delDifEc} and since $(1+z)P_{n-1}(z)$ is zero at $z=-1$ and at the zeros of $P_{n-1}(z)$ by Rolle's theorem and induction the statements on the zeros follow immediately. 

2. Because of
\[
f(-1-z,-t)=e^{t}f(z,t),
\]
doing the Taylor polynomial in $t=0$ of each function in the  above equality and from \eqref{Recurrencia} we get \eqref{simtRel}. The statement about the symmetric relation for the zeros of $P_n$ follows from \eqref{simtRel}. 
Taking derivatives in \eqref{simtRel} we get \eqref{RelPolDer}. The relation \eqref{RelPolDer2} follows immediately from \eqref{simtRel} and \eqref{RelPolDer}.
\end{proof}

\begin{rem} Formula \eqref{delDifEc} appears in \cite{Dil_Kur}. The simplicity and interlace of the zeros of consecutive Eulerian polynomials is well known \cite{SavVis} and it can be deduced from the previous lemma and \eqref{conecEulerianPolvsGeomPol}.

\end{rem}

The symmetry property \eqref{simtRel} of the polynomials $P_n$ with respect to $-1/2$ implies several orthogonality relations for the Geometric polynomials.

\begin{lemma}  \label{lemOrt1} 
\begin{enumerate}
\item If $m$ and $n$ are positive integers of different parity we have
\[
\int_{-1}^0\left(\frac{P_{2j}(x)}{x}\right)^m\left(\frac{P_{2k}(x)}{x}\right)^n\, \dd x=0,\quad j,k\in\N,\, j\ne k.
\]
\item It holds
\[
\int_{-1}^0\frac{P_{2j}(x)}{x}\frac{P_{2k-1}(x)}{x}\, \dd x=0,\quad k,j\ge 1.
\]
\item If $n\ge 2$, we have
\begin{equation}
\label{ort1}
\int_{-1}^0\frac{P_n(x)}{x}\dd x=0.
\end{equation}
\end{enumerate} 
\end{lemma}

\begin{proof} The first two statements follow from the symmetry property of the polynomials $P_n$ with respect to $-1/2$ given by \eqref{simtRel}.

From \eqref{delDifEc} we have
\[
\int_{-1}^0\frac{P_n(x)}{x}\dd x=\left.(1+x)P_{n-1}(x)\right|_{x=0}-\left.(1+x)P_{n-1}(x)\right|_{x=-1}=0,\quad n\ge 2.\qedhere
\]
\end{proof}

\begin{rem} Statement 3 of Lemma \ref{lemOrt1} can be also proved from the relation
$f(z,t)-1=z(1-e^t)f(z,t).$
\end{rem}

Let $(z_{k,n}^*)$ denote the $n-1$ zeros of $P_n$ different from zero numbered according to their increasing order.

\begin{lemma}\label{LemCocienteCons} There exist $n-1$ positive numbers $\lambda_{k,n}$ such that 
\begin{equation}
\label{descFracSimples}
n\frac{P_{n-1}(z)}{P_n(z)}=\sum_{k=1}^{n-1}\frac{\lambda_{k,n}}{z-z_{k,n}^*}
\end{equation}
and $\sum_{k=1}^{n-1}\lambda_{k,n-1}=1$. Moreover the sequence of functions $\left(n\frac{P_{n-1}(z)}{P_n(z)}\right)_{n\in\mathbb{N}}$ is uniformly bounded in compact subsets of  $\mathbb{C}\setminus[-1,0]$; i.e. if $K\subset \overline{\mathbb{C}}\setminus[-1,0]$ is a compact set, then there exists a constant $C=C(K)>0$ such that for all  $z\in K$ we have
\[
\left|n\frac{P_{n-1}(z)}{P_n(z)}\right|\le C.
\]

\end{lemma}

\begin{proof} According to  Lemma \ref{lemZer} the zeros of $P_n$ are simple and lie in $(-1,0]$. Let
\[
Q_k(z)=\frac{P_k(z)}{z},\quad k=1,2,\ldots
\]
The zeros of $Q_n$ are the $n-1$ numbers $(z_{k,n}^*)$. By the Lagrange  interpolation formula we have
\[
nQ_{n-1}(z)=\sum_{k=1}^{n-1}nQ_{n-1}(z_{k,n}^*)\frac{Q_n(z)}{Q_n'(z_{k,n}^*)(z-z_{k,n}^*)}.
\]
Thus, the equality \eqref{descFracSimples} follows with
\[
\lambda_{k,n}=\frac{nQ_{n-1}(z_{k,n}^*)}{Q_n'(z_{k,n}^*)}.
\]
The polynomials $Q_n$ y $Q_{n-1}$ have degree exactly $n-1$ and $n-2$, respectively, with leading coefficients $n!$ and $(n-1)!$ Multiplying  \eqref{descFracSimples} by $z$ and taking limits as $z\to\infty$, we obtain
\[
\sum_{k=1}^{n-1}\lambda_{k,n}=1.
\]
According to  Lemma \ref{lemZer}, the zeros of $Q_n$ and $Q_{n-1}$  interlace in $(-1,0)$. From Rolle's theorem $Q_n$ and $Q_n'$ have also interlacing zeros, so 
\[
Q_n'(z_{k,n}^*)Q_n'(z_{k-1,n}^*)<0,
\]
the same happen with $Q_{n-1}$. As $Q_{n-1}$ and $Q_{n}$ have positive leading coefficient, we have
$Q_n'(z_{k,n}^*)Q_n(z_{k,n}^*)>0$ and hence $\lambda_{k,n}>0$, $k=1,2,\ldots,n-1.$

If $K\subset \mathbb{C}\setminus [-1,0]$ is a compact set, 
\[
d(K,[-1,0])=\inf\{|z-x|; x\in[-1,0],\, z\in K\}>0.
\]
From \eqref{descFracSimples} it follows
\[
\left|n\frac{P_{n-1}(z)}{P_n(z)}\right|\le \frac{1}{d(K,[-1,0])}.\qedhere
\]

\end{proof}

Next we translate the above results to Eulerian polynomials. It is known (see, for example, \cite[p. 51]{Com}, \cite{Foa}, \cite{Lus} and \cite{Qi} and the references therein) the exponential generating function relation
  \begin{equation}\label{eq:eulerianGF}
  \sum_{n=0}^{\infty} A_n(x)\,\frac{t^n}{n!} = \frac{1-x}{e^{t(x-1)}-x},\qquad x\ne 1,
  \end{equation}
the recurrence-differential relation, $A_0(x)=1$,
\[
  A_{n+1}(x) = (1+n x)\,A_n(x) + x(1-x)\,A_n'(x),\quad n=0,1,\ldots,
\]
and the palindromic relation
\begin{equation}\label{eq:pal}
 x^{n-1}\,A_n(1/x)=A_n(x),\quad n\ge 1, \quad x\ne 0.
\end{equation}
The first seven Eulerian Polynomails are
\[
A_0(x)=1,\quad A_1(x)=1,\quad A_2(x)=1+x,\quad A_3(x)=1 + 4 x+ x^2,
\]
\[
A_4(x)=1 + 11 x + 11 x^2 + x^3, \quad A_5(x)=1 + 26 z + 66 x^2 + 26 x^3 + x^4,
\]
and
\[
A_6(x)=1 + 57 x + 302 x^2 + 302 x^3 + 57 x^4 + x^5.
\]
Note that we have a shift in the degree (see \cite[pp. 485-486]{Eul}). Differentiating \eqref{eq:pal} yields identities which are the exact analogues of the derivative relations for $P_n$.
\begin{lemma}\label{lem:derivatives}
For $x\ne 0$,
\begin{equation}
 x\,A_n'(x) + x^{n-2}\,A_n'(1/x) = (n-1)\,A_n(x),\label{eq:der-one}
 \end{equation}
 and
 \begin{equation}
 \frac{A_n'(x)}{A_n(x)} + \frac{1}{x^2}\,\frac{A_n'(1/x)}{A_n(1/x)} = \frac{n-1}{x},\label{eq:der-two}
\end{equation}
where in the last relation we have to assume that $A_n(x)\ne 0$.
\end{lemma}

Next we stated the corresponding orthogonality relation for Eulerian polynomials. From \eqref{conecEulerianPolvsGeomPol} and using the change of variables
\begin{equation}\label{eq:xu}
 x=1+\frac{1}{u},\quad\quad u=\frac{1}{x-1}, \qquad \dd u = -\frac{\dd x}{(x-1)^2},
\end{equation}
for positive integers $r_\ell$ and nonnegative integer exponents $\alpha_\ell$, $1\le \ell\le p$, we obtain
\begin{align}
 \int_{-1}^0 \prod_{\ell=1}^p \left(\frac{P_{r_\ell}(u)}{u}\right)^{\alpha_\ell}\,\dd u
 &= \int_{-\infty}^{0} \frac{\prod_{\ell=1}^p A_{r_\ell}(x)^{\alpha_\ell}}{(1-x)^{E}}\,\dd x,
 \qquad \label{eq:transfer}
\end{align}
where $E:=2+\sum_{\ell=1}^p \alpha_\ell(r_\ell-1).$ We use $(x-1)^{-(\cdot)}=(\pm1)(1-x)^{-(\cdot)}$ and absorb the sign, irrelevant to vanishing. Hence Lemma \ref{lemOrt1} is written
\begin{lemma}\label{lem:Euler-analogue}
\hfill
\begin{enumerate}
  \item If $m$ and $n$ are positive integers of different parity and $j\ne k$ we have 
  \begin{equation}\label{eq:analogue1}
  \int_{-\infty}^{0} \frac{A_{2j}(x)^m\,A_{2k}(x)^n}{(1-x)^{\,2+m(2j-1)+n(2k-1)}}\,\dd x=0.
  \end{equation}

  \item For $j,k\ge1$,
  \begin{equation}\label{eq:analogue2}
  \int_{-\infty}^{0} \frac{A_{2j}(x)\,A_{2k-1}(x)}{(1-x)^{\,2+(2j-1)+(2k-2)}}\,\dd x=0.
  \end{equation}

  \item For $n\ge2$,
  \begin{equation}\label{eq:analogue3}
  \int_{-\infty}^{0} \frac{A_n(x)}{(1-x)^{\,n+1}}\,\dd x=0.
  \end{equation}
\end{enumerate}
\end{lemma}
On the analogous to the ratio relation \eqref{descFracSimples} for Eulerian polynomial we write the zeros of $A_n$ as $\zeta_{1,n}<\cdots<\zeta_{n-1,n}<0$. It is known that they are all real, simple and negative, and that the zeros of $A_{n-1}$ interlace those of $A_n$ (see \cite{SavVis,YanZha}). By \eqref{conecEulerianPolvsGeomPol}, the relation \eqref{descFracSimples}
 is written as
\begin{lemma}\label{lem:PF}
There exist positive numbers $\alpha_{k,n}$, $k=1,\dots,n-1$, such that $\sum_{k=1}^{n-1}\alpha_{k,n}=1$ and
\begin{equation}\label{eq:A-PF}
 n\frac{A_{n-1}(x)}{A_n(x)} = \sum_{k=1}^{n-1}\frac{\alpha_{k,n}}{x-\zeta_{k,n}}, 
\end{equation}
 $\zeta_{k,n}$ are the simple real zeros of $A_n$, all lying in $(-\infty,0)$. Moreover the sequence of functions $\left(\frac{n\,A_{n-1}(z)}{A_n(z)}\right)_{n\in\mathbb{N}}$ is uniformly bounded in compact subsets of  $\mathbb{C}\setminus(-\infty,0]$; i.e. if $K\subset \mathbb{C}\setminus(-\infty,0]$ is a compact set, then there exists a constant $C=C(K)>0$ such that for all  $z\in K$ we have
\[
\left|\frac{n\, A_{n-1}(z)}{A_n(z)}\right|\le C.
\]
\end{lemma}
In the above lemma one has $\alpha_{k,n}:=\frac{\lambda_{k,n}}{-x_{k,n}} >0$.

\section{Proof of main results}\label{SectProofMain-Result}

In this section we prove Theorem \ref{teoAsympIn} and others about the  asymptotic behaviour of Geometric polynomials. Also we provide the corresponding results for Eulerian polynomials.

\begin{theo}\label{teoAsymp} We have
\begin{equation}
\label{convForm}
\lim_{n\to\infty}\frac{P_n(z)(\log\frac{1+z}{z})^{n+1}}{n!}=\frac{1}{1+z}
\end{equation}
uniformly on compact subsets of $\overline{\C}\setminus [-1,0].$

\end{theo}

Hereafter $\log(\cdot)$ is the principal branch of the logarithmic. For $z>0$ Theorem \ref{teoAsymp} is proved  in \cite{Bel}.

\begin{proof}
The function $f(z,t)$ has a simple pole at each point $t=t_k=\log\frac{1+z}{z}+2\pi k i,$ $k\in\Z$. These poles are simple and
\[
\lim_{t\to \log\frac{1+z}{z}}(t-\log\frac{1+z}{z})f(z,t)=\frac{-1}{1+z}.
\]
Let $L(z)=\min\{|\log\frac{1+z}{z}+2\pi i|, |\log\frac{1+z}{z}-2\pi i|\}$. The principal branch of the logarithmic function $\log\frac{1+z}{z}$ has imaginary part in $(-\pi,\pi)$,
\[
\log\frac{1+z}{z}=\log\left|\frac{1+z}{z}\right|+i \arg\left(\frac{1+z}{z}\right),\quad  
\arg\left(\frac{1+z}{z}\right)\in (-\pi,\pi)
\]
for $z\in \C\setminus [-1,0]$. Observe that $\log\frac{1+z}{z}+2\pi i$ and $\log\frac{1+z}{z}-2\pi i$ are the  singularities  of $f(z,t)$  different from $\log\frac{1+z}{z}$ with smaller modulus. Then
\[
\left|\log\frac{1+z}{z}\right|<L(z)
\]
and we can take $R=R(z)$ such that $|\log\frac{1+z}{z}|<R<L(z)$. 
Thus, 
\[
g(z,t)=f(z,t)+\frac{1}{(z+1)(t-\log\frac{1+z}{z})}
\]
is an analytic function in the disk
\[
\{t\in\C:|t|<L(z)\},
\]
which contains to $t=\log\frac{1+z}{z}$. Then
\begin{align}
\label{eq2}
\frac{P_n(z)}{n!}&=\frac{1}{2\pi i}\oint_{|t|=R}\frac{g(z,t)}{t^{n+1}}\, \dd t-\frac{1}{2\pi i}\oint_{|t|=R}\frac{\dd t}{(1+z)(t-\log\frac{1+z}{z})t^{n+1}} \nonumber\\&=\frac{1}{2\pi i}\oint_{|t|=R}\frac{g(z,t)}{t^{n+1}}\, \dd t+\frac{1}{(1+z)(\log\frac{1+z}{z})^{n+1}}.
\end{align}
Since $g(z,t)$ is an analytic function in $\{t\in\C:|t|<L(z)\}$ we have
\begin{equation}
\label{oPeq}
\frac{1}{2\pi i}\oint_{|t|=R}\frac{g(z,t)\dd t}{t^{n+1}}=O\left(\frac{1}{R^n}\right)=o\left(\frac{1}{|\log\frac{1+z}{z}|^n}\right),\, \text{as }n\to\infty,
\end{equation}
and then from \eqref{eq2} follows \eqref{convForm}.

Now let $K$ be a compact set such that $K\subset \C\setminus [-1,0]$. Since the translation and the principal branch of logarithm are continuous functions, for each point $z_0\in K$ we can find an open neighborhood $V(z_0)$ of $z_0$ such that for all $z\in \overline{V(z_0)}$ we have
\[
\left|\log\frac{1+z}{z}\right|<R(z_0)<L(z).
\]
Thus \eqref{oPeq} holds uniformly for $z\in V(z_0)$. As $K$ is a compact subset we can find a finite number of sets $\left(V(z_j)\right)_{j=1}^{N}$ which cover $K$ such that  
\begin{equation}
\label{eqConvUnif}
\lim_{n\to\infty}(R(z_j))^n\left(\frac{P_n(z)}{n!}-\frac{1}{(1+z)(\log\frac{1+z}{z})^{n+1}}\right)=0,
\end{equation}
$z\in V(z_j)$ and $j=1,\ldots,N$. Therefore \eqref{convForm} holds uniformly on compact subsets of $\C\setminus [-1,0]$. 

Finally we show that \eqref{convForm} is also true on compact subset of $\overline{\C}\setminus [-1,0]$. For each compact set $K\subset \overline{\C}\setminus [-1,0]$, there exists $R>0$ such that    $K\subset \Delta_R=\{z\in\overline{\C}\setminus[-1,0]:\frac{1}{|\log\frac{1+z}{z}|}\ge R\}$. Since  $\log\frac{1+z}{z}\sim \frac1z$ as $z\to\infty$, the function $\frac{P_n(z)(\log\frac{1+z}{z})^{n+1}}{n!}-\frac{1}{z+1}$ admit an analytic extension to $\overline{\C}\setminus [-1,0]$. In addition, we have proved that \eqref{convForm} holds uniformly on the level curve $\{z\in\C\setminus[-1,0]:\frac{1}{|\log\frac{1+z}{z}|}=R\}$. Hence, by the maximum principle
\[
\lim_{n\to\infty}\max_{z\in \Delta_R}\left|\frac{P_n(z)(\log\frac{1+z}{z})^{n+1}}{n!}-\frac{1}{z+1}\right|=0.\qedhere
\]
\end{proof}

\begin{figure}
\centering
\begin{subfigure}[b]{0.45\textwidth}
\includegraphics[width=\textwidth]{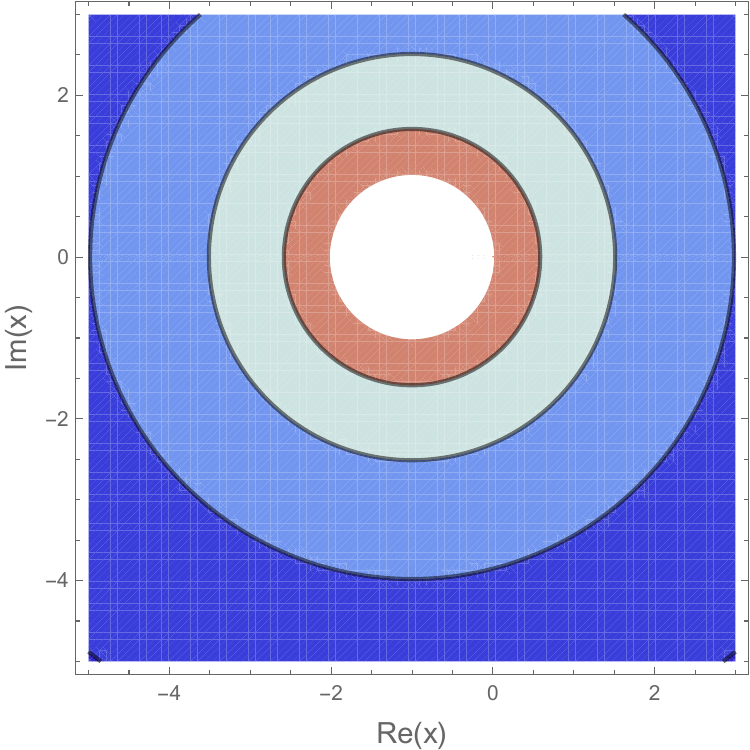}
\end{subfigure}
\begin{subfigure}[b]{0.45\textwidth}
\includegraphics[width=\textwidth]{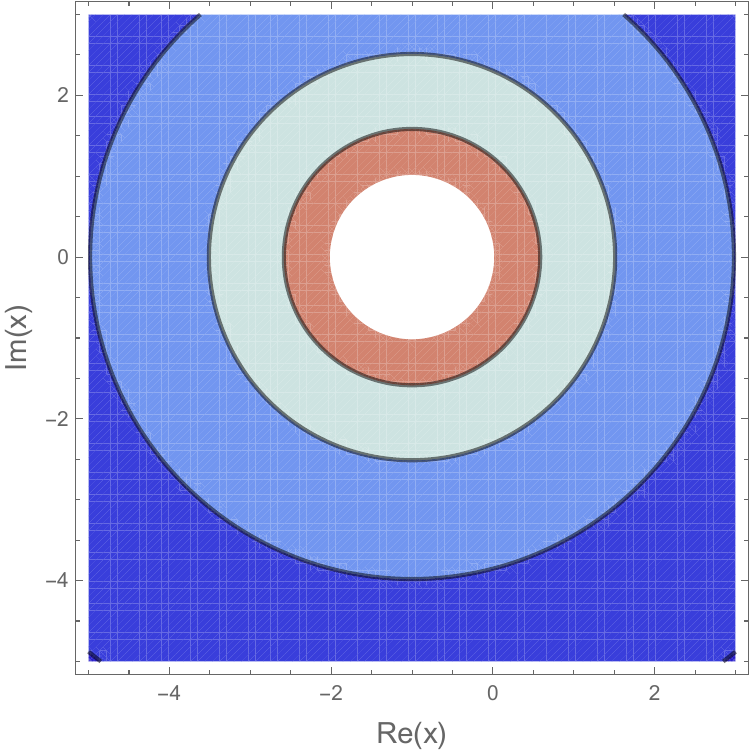}
\end{subfigure}
\caption{Level curve  of the modulus of $\frac{P_n(z)(\log\frac{1+z}{z})^{n+1}}{n!}$ (left) and its approximation, $\frac{1}{1+z}$ (right).}\label{fig:ModPolApproxOut}
\end{figure}

\begin{figure}
\centering
\begin{subfigure}[b]{0.45\textwidth}
\includegraphics[width=\textwidth]{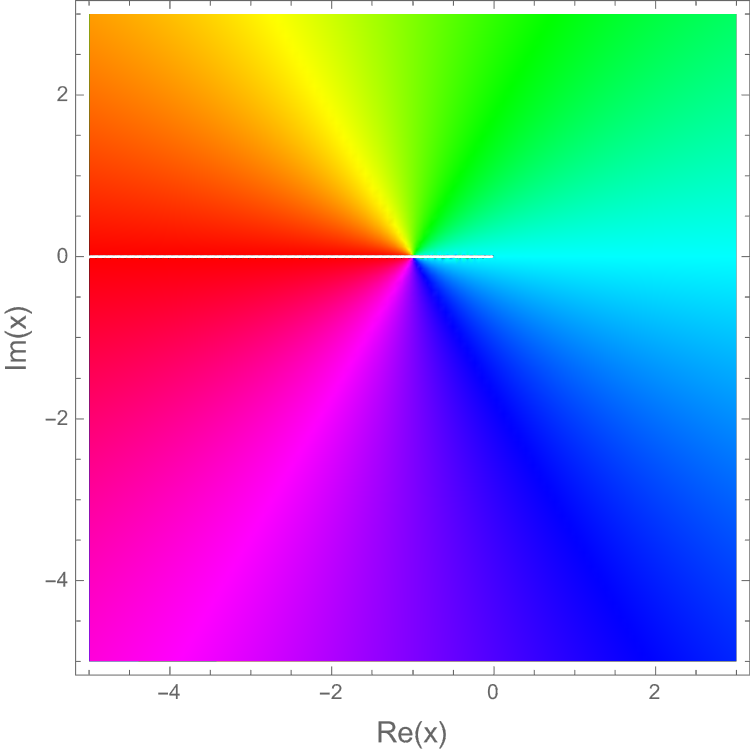}
\end{subfigure}
\begin{subfigure}[b]{0.45\textwidth}
\includegraphics[width=\textwidth]{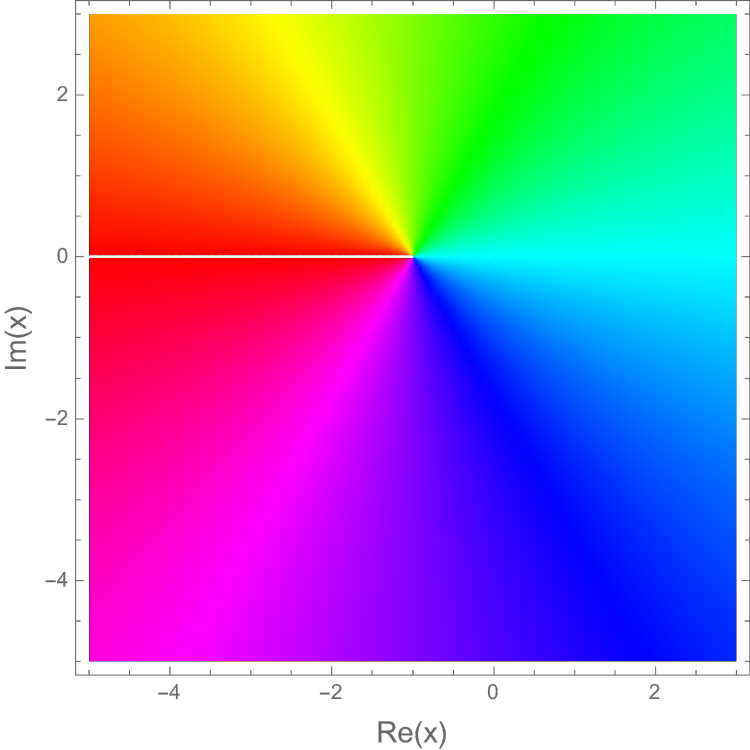}
\end{subfigure}
\caption{Hue of the arguments for $\frac{P_n(z)(\log\frac{1+z}{z})^{n+1}}{n!}$ (left) and its approximation, $\frac{1}{1+z}$ (right).}\label{fig:ArgPolApproxOut}
\end{figure}

\begin{rem}\label{obsGemConv} According to \eqref{eqConvUnif} the rate of converge in Theorem  \ref{teoAsymp} is geometric on each compact subset of $\overline{\C}\setminus[-1,0]$; namely, for each compact subset $K\subset \overline{\C}\setminus[-1,0]$ there exist positive constants $q,C$ with  $q<1$ such that
\[
\max_{z\in K}\left|\frac{P_n(z)\left(\log\frac{1+z}{z}\right)^{n+1}}{n!}-\frac{1}{1+z}\right|\le Cq^n.
\]
To give an idea of the proximity of the sequence and its limit, Figures \ref{fig:ModPolApproxOut} and \ref{fig:ArgPolApproxOut} show modulus and hue of the arguments for $\frac{P_n(z)(\log\frac{1+z}{z})^{n+1}}{n!}$ and its approximation, $\frac{1}{1+z}$, respectivaly.

\end{rem}

Let $\mu_{P_n}$ be the normalized counting measure on the zeros of $P_n$, i.e. $\mu_{P_n}=\frac1n\sum_{j}\delta_{z_{j,n}}$, where $\delta_{z_{j,n}}$ is the Dirac measure at $z_{j,n}$. Because of
\[
\frac{nP_{n-1}(z)}{P_n(z)}=\frac{\frac{P_{n-1}(z)}{(n-1)!}(\log\frac{1+z}{z})^n}{\frac{P_{n}(z)}{n!}(\log\frac{1+z}{z})^{n+1}}\log\frac{1+z}{z},
\]
from Theorem \ref{teoAsymp} follows the next result.

\begin{cor}\label{corLimRaizNesima} We have
\begin{equation*}
\label{limCociente}
\lim_{n\to\infty}\frac{nP_{n-1}(z)}{P_n(z)}={\log\frac{1+z}{z}},
\end{equation*}
uniformly on compact subsets of $\overline{\C}\setminus [-1,0]$. Moreover, 
\begin{equation*}
\label{corRaizLim}
\lim_{n\to\infty}\left|\frac{P_n(z)}{n!}\right|^{1/n}=\frac{1}{|\log\frac{1+z}{z}|}
\end{equation*}
and 
\begin{equation*}
\label{LimPotenciales}
\lim_{n\to\infty}\int\log|z-t|\, d\mu_{P_n}(t)=-\log \left|\log\frac{1+z}{z}\right|
\end{equation*}
uniformly on compact subsets of $\C\setminus [-1,0]$.
\end{cor}

\begin{rem} According to \eqref{FubNum} and Theorem \ref{teoAsymp}, we have the following relation which can be also found in \cite{Gro},
\[
\lim_{n\to\infty}\left(\frac{(\log 2)^{n+1}}{n!}\sum_{k=0}^\infty\frac{k^n}{2^k}\right)=1.
\]
\end{rem}

The asymptotics of $(P_n(z))_{n\ge 0}$ in the interval $(-1,0)$ given in Theorem \ref{teoAsympIn} requires to observe that the function $f(z,t)$ has a simple poles at $\ell(z)\pm i(2k-1)\pi$, $k\in\Z$, with residue at these points $-1/(1+z)$ for each $z\in (-1,0)$. Remember that $\ell(z)=\log|\frac{1+z}{z}|$. Before giving the proof of Theorem \ref{teoAsympIn} we  state a technical results. Define the function
\[
h(z,t)=(1+z)\left(f(z,t)+\frac{1}{1+z}\left(\frac{1}{t-(\ell(z)+\pi i)}+\frac{1}{t-(\ell(z)-\pi i)}\right)\right).
\]
This function is analytic in $\{t\in\C:|t|<|\ell(z)+3\pi i|\}$. Doing the change of variable $u=u(t,z)=t-\ell(z)$ and by the partial fraction decomposition of hyperbolic tangent, we have
\begin{equation}
\label{HiperbFractionDecomposition}
h(z,t)=G(u)=\frac{1}{e^u+1}+\frac{1}{u-\pi i}+\frac{1}{u+\pi i}=\frac12-2u\sum_{k=2}^\infty\frac{1}{u^2+(2m-1)^2\pi^2}.
\end{equation}
\begin{lemma} \label{lemAcotaFuncSinSingularidad}There exist positive constants\footnote{As usual, the same constants in different expressions may not represent the same values. } $C_1,C_2$ such that for all $z\in(-1,0)$ 
\[
\sup\{|G(u)|:|u+\ell(z)|=\sqrt{(\ell(z))^2+2\pi^2}\}\le C_1+C_2|\ell(z)|.
\] 
\end{lemma}
\begin{proof} By the triangular inequality, if $|u+\ell(z)|=\sqrt{(\ell(z))^2+2\pi^2}$, we have
\[
|u|\le |\ell(z)|+\sqrt{(\ell(z))^2+2\pi^2}\le C_1|\ell(z)|+C_2,
\]
for some positive constants $C_1,C_2$. Since $\ell(z)-i(2m-1)\pi$ and $\ell(z)+i(2m-1)\pi$ are in the same line of $\ell(z)$, the center of the circle $|u+\ell(z)|=\sqrt{(\ell(z))^2+2\pi^2}$, we have
\[
|u^2+(2m-1)^2\pi^2|=|u-i(2m-1)\pi||u+i(2m-1)\pi|\ge(2m-1)^2\pi^2
\]
for  all such $u$ and the statement of the lemma follows immediately from \eqref{HiperbFractionDecomposition} and the above two inequalities.
\end{proof}

\begin{proof}[Proof of Theorem \ref{teoAsympIn}] Let
\[
H_n(z)=\frac{(1+z)P_n(z)((\ell(z))^2+\pi^2)^{(n+1)/2}}{n!}.
\]
Observe that
\begin{equation}
\label{logSeq}
\ell(z_n)=-L_n\quad \text{and}\quad |\ell(z)|\le |\ell(z_n)|,\quad z\in  [z_n,-1/2].
\end{equation}
 We have
\begin{equation}\label{FormInt}
H_n(z)-J_n(z)=\frac{((\ell(z))^2+\pi^2)^{(n+1)/2}}{2\pi i}\oint_{|t|=\sqrt{(\ell(z))^2+2\pi^2}}\frac{h(z,t)}{t^{n+1}}\, dt.
\end{equation}
The left hand side of the above equation is symmetric with respect to $z=-1/2$, so its enough to check the uniform convergence for $z\in [z_n,-1/2]$.
From Lemma \ref{lemAcotaFuncSinSingularidad}, \eqref{HiperbFractionDecomposition}, \eqref{logSeq}, and \eqref{FormInt} we get that there exists positive constant $C_1,C_2$ such that
\begin{align}\notag
\left|H_n(z)-J_n(z)\right| &\le (C_1|\ell(z)|^2+C_2)\left(\frac{(\ell(z))^2+\pi^2}{(\ell(z))^2+2\pi^2}\right)^{n/2}\\ &\le (C_1|\ell(z_n)|^2+C_2)\left(\frac{(\ell(z_n))^2+\pi^2}{(\ell(z_n))^2+2\pi^2}\right)^{n/2}.\label{AcotUnif}
\end{align}
Since $\log(1-x)<-x$ for $x\in (0,1)$, we obtain
\begin{align}
\left(\frac{(\ell(z_n))^2+\pi^2}{(\ell(z_n))^2+2\pi^2}\right)^{n/2}&=e^{\frac{n}{2}\notag\log\left(1-\frac{\pi^2}{(\ell(z_n))^2+2\pi^2}\right)}\\ &<e^{\frac{-\pi^2 n}{2(\ell(z_n))^2+4\pi^2}}=n^{-\frac{\pi^2}{2c^2}+o(1)},\quad \text{as }n\to\infty.\label{AcotExp}
\end{align}
Therefore, the theorem follows from \eqref{DefSeq}, \eqref{logSeq}, \eqref{AcotUnif}, and \eqref{AcotExp}.
\end{proof}

Since $c$ can be any number in $(0,\frac{\pi}{\sqrt{2}})$ the inequalities \eqref{AcotUnif} and \eqref{AcotExp} let us sharp the statement of Theorem \ref{teoAsympIn}.

\begin{theo}  \label{teoAsympIn2}  Given any $\delta>0$ there exists a  sequence $(z_n)_{n\in\N}\subset (-1,0)$, $\lim_{n\to\infty}z_n=-1$ such that
\[
\frac{(1+z)P_n(z)((\ell(z))^2+\pi^2)^{(n+1)/2}}{n!}-J_n(z)=o(\frac1{n^{\delta}})
\]
uniformly for $z\in[z_n,-1-z_n]$ as $n\to\infty$.
\end{theo}

Taking into account Theorem \ref{teoAsympIn}  and $|J_n(z)|\le 2$ with $z\in(-1,0)$, for each $\epsilon>0$ there exists $N_0\in\N$ such that
\[
|P_n(z)|\le (2+\epsilon)\frac{n!}{(1+z)((\ell(z))^2+\pi^2)^{(n+1)/2}},\quad \text{for all }n\ge N_0,\quad z\in[z_n,-1-z_n].
\]
\begin{cor}Let $(z_n)_{n\in\N}$ be the sequence of Theorem \ref{teoAsympIn}. Then there exists a constant $C>0$ such that for all $n\in\N$ it holds
\[
|P_n(z)|\le C\frac{n!}{(1+z)((\ell(z))^2+\pi^2)^{(n+1)/2}}, \quad z\in [z_n,-1-z_n].
\]
\end{cor}

In the Figure \ref{fig:PolApprox} appears the graphs of $\frac{(z+1)P_{25}(z)((\ell(z))^2+\pi^2)^{13}}{25!}$ and its approximation, $\frac{(\ell(z)+\pi i)^{26}+(\ell(z)-\pi i)^{26}}{((\ell(z))^2+\pi^2)^{13}}$. We do not show both graphs in the same Cartesian coordinate system because they are strikingly similar.

Additionally, to obtain an asymptotic behavior for the zeros of $P_n$ in the bulk of the interval $(-1,0)$ in the next section we need the following result about $P_n'$. We have
\[
\sum_{k=0}^\infty P_n'(z)\frac{t^k}{k!}=\frac{e^t-1}{(1-z(e^t-1))^2}.
\]
For $z\in(-1,0)$ this generating function, as function of $t$, has a poles of order 2 at $t=\ell(z)\pm \pi i$. The principal part of its Laurent series at these points are 
\[
\frac{1}{z(1+z)^2}\frac{1}{\left(t-(\ell(z)\pm \pi i)\right)^2}+\frac{1}{(1+z)^2}\frac{1}{\left(t-(\ell(z)\pm \pi i)\right)}.
\] 
Then following step by step the proof of Theorem \ref{teoAsympIn2} we obtain the following asymptotic behavior for $P_n'$ in $(-1,0)$. 

\begin{theo}  \label{teoAsympDer}  There exists a  sequence $(w_n)_{n\in\N}\subset (-1,0)$, $\lim_{n\to\infty}w_n=-1$ such that
\[
\frac{z(1+z)^2P_n'(z)((\ell(z))^2+\pi^2)^{(n+2)/2}}{n!\, n}-J_{n+1}(z)=o(\frac1{n})
\]
uniformly for $z\in[w_n,-1-w_n]$ as $n\to\infty$.
\end{theo}

\begin{figure}
\centering
\begin{subfigure}[b]{0.45\textwidth}
\includegraphics[width=\textwidth]{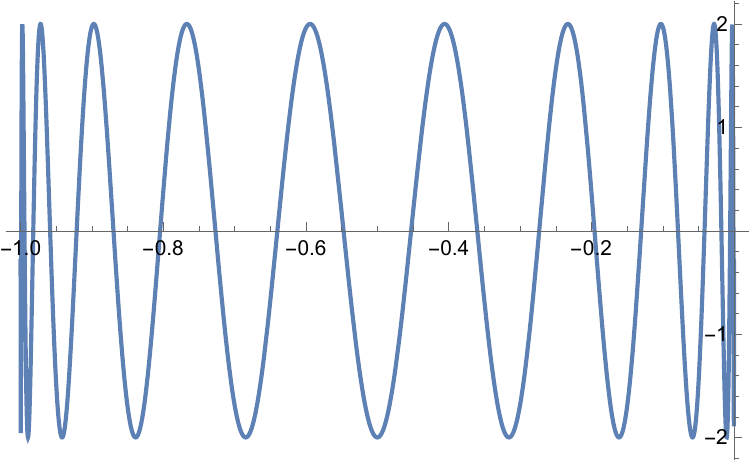}
\end{subfigure}
\begin{subfigure}[b]{0.45\textwidth}
\includegraphics[width=\textwidth]{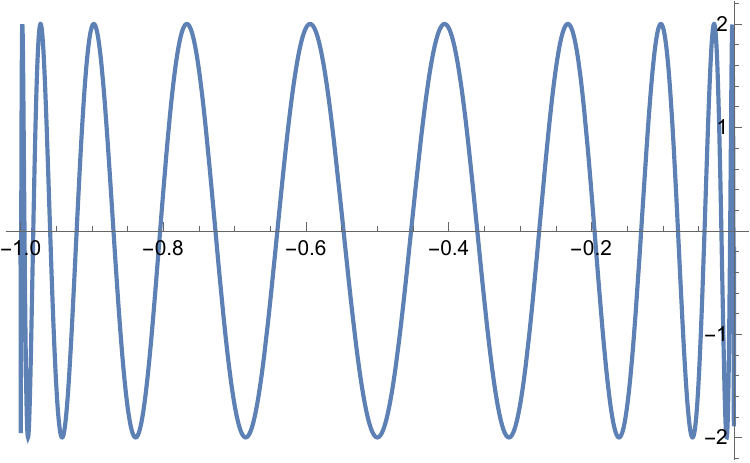}
\end{subfigure}
\caption{Graphics  of $\frac{(z+1)P_{25}(z)((\ell(z))^2+\pi^2)^{13}}{25!}$ (left) and its approximation, $\frac{(\ell(z)+\pi i)^{26}+(\ell(z)-\pi i)^{26}}{((\ell(z))^2+\pi^2)^{13}}$ (right).}\label{fig:PolApprox}
\end{figure}

The following results are analogous of the previous ones  for Eulerian polynomials.
\begin{theo} We have
\begin{equation}\label{eq:euler-limit-x} 
\lim_{n\to\infty}\frac{A_n(z)}{n!}\,\frac{(\log z)^{n+1}}{(z-1)^n}=\frac{z-1}{z}
\end{equation}
uniformly on compact subsets of $\overline{\C}\setminus(-\infty,0]$
\end{theo}

In \eqref{eq:euler-limit-x} the value at $z=1$ of the expression in the limit is zero because for all $n\in\N$
\[
\lim_{z\to 1}\frac{A_n(z)}{n!}\,\frac{(\log z)^{n+1}}{(z-1)^n}=0.
\]
All analogous expression for Eulerian polynomials are understood that the value at $z=1$ is defined by the corresponding limit. Let $\mu_{A_n}$ be the normalized counting measure on the zeros of polynomial $A_n$.

\begin{cor} \label{LimCocRaizEulerian} We have
\[
\lim_{n\to\infty}\frac{nA_{n-1}(z)}{A_n(z)}=\frac{\log z}{z-1},
\]
\[
\lim_{n\to\infty}\left|\frac{A_n(z)}{n!}\right|^{1/n}=\left|\frac{z-1}{\log z}\right|,
\]
and 
\[
\lim_{n\to\infty}\int\log|z-t|\, d\mu_{A_n}(t)=\log \left|\frac{z-1}{\log z}\right|.
\]
All above limits  hold uniformly on compact subsets of $\C\setminus (-\infty,0]$.
\end{cor}

Define for $x\in(-\infty,0)$,
\begin{equation}\label{eq:JnEuler}
 \widehat{J}_n(x)=\frac{\big(\log|x|+\pi i\big)^{n+1}+\big(\log|x|-\pi i\big)^{n+1}}{\big(\,(\log|x|)^2+\pi^2\,\big)^{(n+1)/2}}=J_n\!\left(\frac{1}{x-1}\right).
\end{equation}
With $c,\,L_n,\,z_n$ as in Theorem~\ref{teoAsympIn}, define $x_n=-\e^{-L_n}$ and the segment $I_n$ as in \eqref{eq:Xn}. Let
\begin{equation}\label{eq:Xn}
x_n = 1+\frac{1}{z_n} = -\mathrm{e}^{-L_n}\in(-\infty,0).
\end{equation}

\begin{theo}\label{thm:euler}
It holds
\begin{equation}\label{eq:euler-limit}
\lim_{n\to\infty}\Bigg(\frac{x\,A_n(x)\,\big(\,(\log|x|)^2+\pi^2\,\big)^{(n+1)/2}}{n!\,(x-1)^{n+1}}-
 \widehat{J}_n(x)\Bigg)=0
\end{equation}
uniformly for $x\in \Big[\frac1{x_n},\ x_n \Big]$.
\end{theo}

\section{Zeros behaviour}\label{sec:Zeros}

In this section, we demonstrate how Theorems \ref{teoAsympIn2} and \ref{teoAsympDer} let us to prove clock behaviour of the zeros of Geometric polynomial on the bulk of the corresponding interval.

As we stated after Theorem \ref{teoAsympIn}, we number also the zeros of $P_n(z)$ with respect to $-1/2$. The zeros of $P_n$ are denoted by $z_{j,n}$ for $j$ positive or negative, depending on whether they lie to the right or left of $z = -1/2$, starting from the zero nearest to $z = -1/2$ and taking into account that $z_{-j,n}+z_{j,n}=-1$. Similarly, the zeros of $J_n$ are denoted by $\eta_{j,n}$ for $j$ positive or negative, depending on whether they lie to the right or left of $z = -1/2$, starting from the zero nearest to $z = -1/2$. To sharp the information on the zeros of $P_n$ we will need the expression for $J_n(z)$ given in the next lemma which follows from 
\[
\overline{\ell(z)+\pi i}=\ell(z)-\pi i,\quad \left|\frac{(\ell(z)+\pi i)^{n+1}}{((\ell(z))^2+\pi^2)^{(n+1)/2}}\right|=1.
\]

\begin{lemma} \label{lemExpCosJn}We have
\[
J_n(z)=2\cos((n+1)\theta(z)),\quad z\in(-1,0),
\]
where the function $\theta(z)\in(0,\pi)$ is defined by
\[
e^{i\theta(z)}=\frac{\ell(z)+\pi i}{\left((\ell(z))^2+\pi^2\right)^{1/2}},\quad z\in (-1,0),
\]
and
\[
\ell(z)=\pi \cot(\theta(z)),\quad \theta=\theta(z)=\arctan\frac{\pi}{\ell(z)}\in(0,\frac{\pi}2).
\]
\end{lemma}
We require another preliminary result concerning the behaviour of $J_n$ around its zeros, which follows directly from Lemma \ref{lemExpCosJn}.

\begin{lemma} \label{levelSetJ_n} Let $\beta>0$. Let $\eta_{j,n}^{(\beta)}$ be in the interval between $\eta_{j,n}$ and the point closest to it where $|J_n(z)|=1$ and $|J_n(\eta_{j,n}^{(\beta)})|=\frac{1}{n^{\beta}}$. Then
\[
 |\eta_{j,n}^{(\beta)}-\eta_{j,n}|\leq \frac{C}{n^{1+\beta}}.
\]
\end{lemma}

\begin{proof} From Lemma \ref{lemExpCosJn} we have $J_n(z)=2\cos((n+1)\theta)$, $\theta=\theta(z)$, and
\begin{equation}
\label{eqEstimDerZvsTheta}
 \frac{\dd z}{\dd\theta}=-\frac{\pi e^{\pi\cot\theta}}{\sin^2\theta(1+e^{\pi\cot\theta})^2},\quad \left|\frac{\dd z}{\dd\theta}\right|\le \frac{\pi}{4},\quad \theta\in(0,\frac{\pi}{2}).
\end{equation}
Let $\theta_{j,n}$ and $\theta_{j,n}^{\beta}$ be the corresponding value of $\eta_{j,n}$ and $\eta_{j,n}^{(\beta)}$. Then
\[
|\sin((n+1)(\theta_{j,n}-\theta_{j,n}^{\beta}))|=\left|\cos\left((n+1)\theta_{j,n}^{\beta}\right)\right|=\frac{1}{2n^{\beta}}.
\]
Hence there exists $C>0$ such that
\begin{equation}
\label{eqEstimTheta}
|\theta_{j,n}-\theta_{j,n}^{\beta}|=\frac{\arcsin(\frac{1}{2n^{\beta}})}{(n+1)}\le \frac{C}{n^{1+\beta}}.
\end{equation}
Then the lemma follows from the mean value theorem, \eqref{eqEstimDerZvsTheta} and \eqref{eqEstimTheta}.
\end{proof}

We are then in a position to prove the clock behavior of the zeros of $P_n(z)$ in the bulk of $(-1,0)$ of Theorem \ref{teoAproxZerosBulk}.

\begin{proof}[Proof of Theorem \ref{teoAproxZerosBulk}] Let us check only when $n$ is even, the other case is similar. 
According to Lemma \ref{lemExpCosJn} the zeros $\eta_{k,n}$ of $J_n(z)$ lie in $(-1,0)$ and satisfy
\[
e^{i\theta_{k,n}}=\frac{\ell(\eta_{k,n})+i\pi}{(\ell(\eta_{k,n})^2+\pi^2)^{1/2}},
\]
where $\theta_{k,n}=\frac{(2k+1)\pi}{2(n+1)}$, $k=0,1,\ldots,n$. Hence
\[
\ell(\eta_{k,n})=\pi \cot\theta_{k,n},
\]
and
\begin{equation}
\label{eqCero1}
\eta_{k,n}=\Phi(\ell(\eta_{k,n})),
\end{equation}
with
\begin{equation}
\label{eqPhi}
\Phi(t)=-\frac{1}{1+e^t}=-1/2+\frac14t+O(t^3),\quad \text{as}\quad t\to 0.
\end{equation}

As it is mentioned we order the zeros of $J_n$ by $\eta_{k,n}$  for $k$ positive or negative depending on whether they lie to the right or left of $z = -1/2$ stating from $z = -1/2$.  Combining \eqref{eqCero1} and \eqref{eqPhi} we get
\begin{equation}
\label{eqCero2}
\eta_{k,n}=-1/2+\frac{\ell(\eta_{k,n})}4+O((\ell(\eta_{j,n}))^3)=-\frac12 + \frac{\pi^2}{4}\,\frac{k}{n+1}  
  +\ o\!\Big(\frac{1}{n^3}\Big),
\end{equation}
as $n\to\infty$. By Lemma \ref{levelSetJ_n}, in the bulk the solutions of $J_n(z)=\pm 1/n$ lie close to the zeros of $J_n(z)$, with an error of order $1/n^2$. Therefore, Theorem \ref{teoAsympIn2} implies that for $n$ large enough, each $P_n$ has a zero within a neighbourhood of radius $O(1/n^2)$ of each bulk zero of $J_n$.  Moreover, when $J_n(z)=0$, $J_{n+1}(z)=\pm 2\sin\theta(z)$; this together with Theorem \ref{teoAsympDer} implies that $P_n$ is asymptotically strictly monotone in a neighbourhood of radius $O(1/n)$ around each zero of $J_n(z)$ in the bulk. Consequently, each $O(1/n^2)$-neighbourhood of a bulk zero of $J_n$ contains exactly one zero of $P_n(z)$. The result then follows from \eqref{eqCero2}. 
\end{proof}

\begin{figure}
\centering
\begin{subfigure}[b]{0.45\textwidth}
\includegraphics[width=\textwidth]{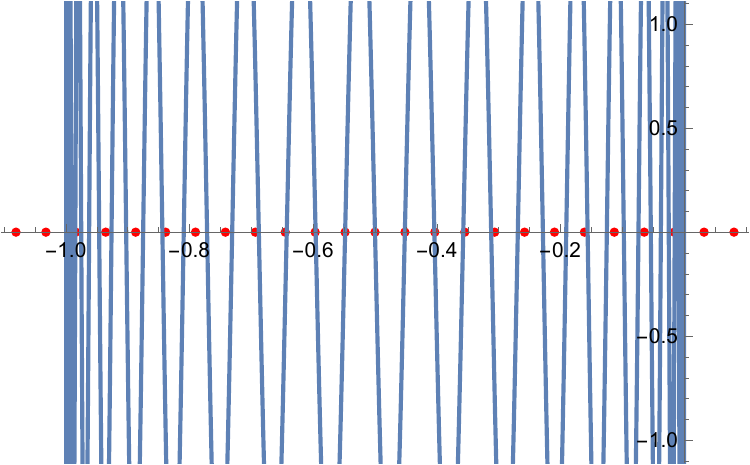}
\end{subfigure}
\begin{subfigure}[b]{0.45\textwidth}
\includegraphics[width=\textwidth]{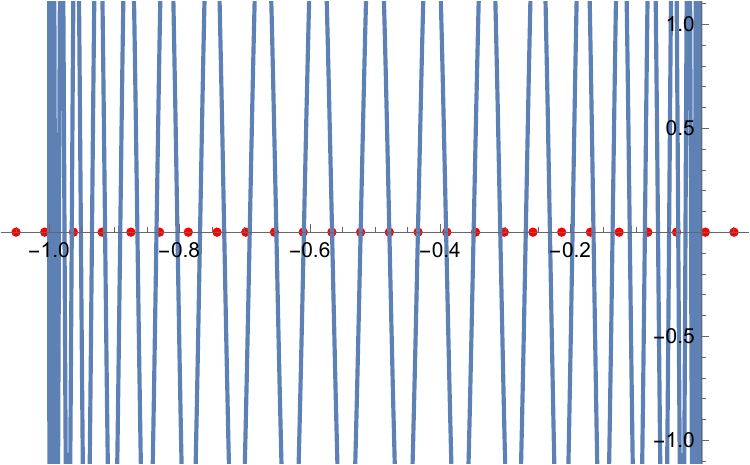}
\end{subfigure}
\caption{ The polynomial $P_{50}$ and, in red, the approximation of its zeros  (on the left). The polynomial $P_{55}$ and, in red, the approximation of its zeros (on the right).}\label{fig:AproxZeros}
\end{figure}

In Figure \ref{fig:AproxZeros} we can see polynomial $P_{50}$, $P_{55}$ and, in red, the approximation of its zeros  according to Theorem \ref{teoAproxZerosBulk}.

Using \eqref{conecEulerianPolvsGeomPol} is straightforward to obtain  the analogous result for the zeros of Eulerian polynomials.

\begin{theo}
\label{thm:bulk-Eulerian}
Let $\zeta_{k,n}$ denote the zeros of $A_n$ ordered by increasing distance from $x=-1$, with the convention $\zeta_{0,2n}=-1$ in the even case. Then as $n\to\infty$ we have the first-order bulk asymptotics:
\begin{align}
  \zeta_{k,2n}^{\pm} &= -1 \pm \pi^2\,\frac{k}{n+1} \; +\; o\!\left(\frac1n\right),\qquad k=0,1,2,\ldots, \label{eq:A-even}
  \\
 \zeta_{k,2n+1}^{\pm} &= -1 \pm \frac{\pi^2}{2}\,\frac{2k-1}{n+1} \; +\; o\!\left(\frac1n\right),\qquad k=1,2,\ldots, \label{eq:A-odd}
\end{align}
where the superscripts $+$ (resp. $-$) refer to zeros in $(-1,0)$ (resp. $(-\infty,-1)$). Moreover, the distance of consecutive zeros in the bulk of $(-\infty,0)$ is
$ \frac{\pi^2}{n} + o\!\left(\frac1n\right).$
\end{theo}

\begin{rem} Analogous results can be obtained  for the polynomials $\mathcal{M}_n(z)$ defined by
\[
 \mathcal{M}_n(z)=(cz+d)^n\,P_n\!\left(\frac{az+b}{cz+d}\right),\quad n\ge0.
\]
where $\gamma=\begin{pmatrix}a&b\\ c&d\end{pmatrix}\in\mathrm{GL}_2(\C)$. The coefficients of $\mathcal{M}_n(z)$ will be integer number if $\gamma\in\mathrm{GL}_2(\Z)$ with $ad-bc=1$.

\end{rem}

\textbf{Acknowledgments.} The authors would like to thank Daniel Perales for letting us know about the paper \cite{JalKabMar}.

\newpage


\noindent
\textbf{M. Bello-Hernández}, corresponding author. Departamento de Matemáticas y Computación, Universidad de La Rioja. c/ Madre de Dios, 53, 26006 Logroño, La Rioja, Spain. Research supported in part by the grant PID2022--138342NB--I00 AEI, from Spanish Government. \\\texttt{mbello@unirioja.es}

\vspace{0.3cm}

\noindent
\textbf{M. Benito}, retired Professor. Instituto P. M. Sagasta, Logro{\~n}o and Departamento de Matemáticas y Computación, Universidad de La Rioja. c/ Madre de Dios, 53, 26006 Logroño, La Rioja, Spain. \\ \texttt{mbenitomunnoz@gmail.com}

\vspace{0.3cm}

\noindent
\textbf{Ó. Ciaurri}. Departamento de Matemáticas y Computación, Universidad de La Rioja. c/ Madre de Dios, 53, 26006 Logroño, La Rioja, Spain. Research supported in part by the grant PID2024-155593NB-C22 AEI, from Spanish Government. \\ \texttt{oscar.ciaurri@unirioja.es}

\vspace{0.3cm}

\noindent
\textbf{E. Fernández}, retired Professor. Instituto P. M. Sagasta, Logro{\~n}o and Departamento de Matemáticas y Computación, Universidad de La Rioja. c/ Madre de Dios, 53, 26006 Logroño, La Rioja, Spain. \\\texttt{emilio.fernandez@unirioja.es}

\end{document}